\documentclass[11pt]{amsart}
\usepackage[T1]{fontenc}
\usepackage[utf8]{inputenc}
\usepackage{amsmath}
\usepackage{amsfonts}
\usepackage{amssymb}
\usepackage{latexsym}
\usepackage{amsthm}
\usepackage[
textwidth=360pt,
textheight=615pt,
hmarginratio=1:1,
vmarginratio=1:1
]{geometry}
\newtheorem{theorem}{Theorem}[section]
\newtheorem{lemma}[theorem]{Lemma}

\newtheorem{coro}[theorem]{Corollary}

\newtheorem{proposition}[theorem]{Proposition}

\newcounter{other}            % Questions get letters
\newtheorem{otherth}[other]{Theorem}              % Other papers' theorems
\newtheorem{otherp}[other]{ Proposition}% Other papers'propositions
\newcommand{\Cn}{\mathbb{C}^n}

\newcommand{\Sn}{\mathbb{S}_ n}

\newcommand{\Bn}{\mathbb{B}_ n}
\newcommand{\D}{\mathbb D}
\newcommand{\T}{\mathbb T}

\numberwithin{equation}{section}
\begin{document}

\title{Tent Carleson measures for  Hardy spaces}

\author[Xiaofen Lv]{Xiaofen Lv}
\address{Xiaofen Lv \\Department of Mathematics\\ Huzhou University, Huzhou
313000,   China } \email{lvxf@zjhu.edu.cn}

\author[Jordi Pau]{Jordi Pau}
\address{Jordi Pau \\Departament de Matem\`{a}tica Aplicada i Analisi\\
Universitat de Barcelona\\ Gran Via 585 \\
08007 Barcelona\\
Spain} \email{jordi.pau@ub.edu}

%\date{}

%\urladdr{}

%
%\subjclass{Primary 30D55; Secondary 46J15}

%\keywords{${H}^p$-spaces, Corona problems, Carleson measure}



\thanks{The first author was partially
supported by  {NSFC (11601149,  11771139) and ZJNSF (LY20A010008)}.
The second author is
 supported by the grants MTM2017-83499-P (Ministerio de Educaci\'{o}n y Ciencia)  and  2017SGR358 (Generalitat de Catalunya). }

%%% ----------------------------------------------------------------------

\begin{abstract}
We completely characterize those positive Borel measures $\mu$ on the unit ball $\Bn$ such that the Carleson embedding from Hardy spaces $H^p$ into the tent-type spaces $T^q_ s(\mu)$ is bounded, for all possible values of $0<p,q,s<\infty$.
\end{abstract}

\maketitle
%%% ----------------------------------------------------------------------
%%% ----------------------------------------------------------------------
\section{Introduction}

The concept of a Carleson measure was introduced by L. Carleson \cite{Carleson-0, carleson}
when studying interpolating sequences for bounded analytic
functions on the unit disk, in route of solving the famous corona problem. Carleson's originally result is a characterization of those positive Borel measures on the unit disk $\D$ for which the embedding $I:H^p\rightarrow L^p(\D,\mu)$ is bounded, for $0<p<\infty$.  Later on,
variations and extensions of the results to the unit ball $\Bn$ of $\Cn$ were obtained by H\"{o}rmander, Duren and Luecking \cite{Du,H,Lue1}, obtaining a full description of the boundedness of the embedding $I:H^p\rightarrow L^q(\Bn,\mu)$ for all possible choices of $0<p,q<\infty$. Moreover, generalizations of the Carleson embedding were subsequently studied by replacing the Hardy space $H^p$ by other function spaces, making Carleson type measures important tools
for the study of modern function and operator theory. We recall that, for $0<p<\infty$, the Hardy space $H^p:=H^p(\Bn)$ consists of those holomorphic functions $f$ in $\Bn$ with
 \[ \|f\|_{H^p}^p=\sup_{0<r<1}\int_{\Sn} \!\! |f(r\zeta)|^p \,d\sigma(\zeta)<\infty,\]
where $d\sigma$ is the surface measure on the unit sphere $\Sn:=\partial \Bn$ normalized so that $\sigma(\Sn)=1$. We refer to the books \cite{Alek}, \cite{Rud} and \cite{ZhuBn} for the theory of Hardy spaces in the unit ball.

Apart from replacing the Hardy space by another function space in the study of the embedding, the next more natural step for generalizing such study, is to replace $L^q(\Bn,\mu)$ by the tent space $T^q_ s (\mu)$, that we are going to define in a moment. We mention here that $L^q(\Bn,\mu)=T^q_ q(\mu)$. As far as we know, tent spaces were introduced by  Coifman, Meyer and Stein \cite{CMS} in order to study several problems in harmonic analysis. These spaces turned to be quite useful in developing further the classical theory of Hardy spaces, closely related to tent spaces due to the Calderon's area theorem. This motivates the study of the boundedness of the embedding $I:H^p\rightarrow T^q_ s(\mu)$ for all possible choices of $0<p,q,s<\infty$. Recently, natural analogues of tent spaces for Bergman spaces have been introduced \cite{PR2,PRS}, playing a similar role for the theory of weighted Bergman spaces as the original ones in the Hardy space. Other interesting results on Carleson measures related to Hardy spaces can be found in \cite{Chang, RFef, L-RP, SaSa}.

In order to define the tent spaces $T^q_ s(\mu)$, we need to introduce first the Kor\'anyi admissible approach region $\Gamma(\zeta)$, defined for $\zeta \in \Sn$ and $\gamma>1$ as
\[
\Gamma(\zeta):=\Gamma_\gamma(\zeta)=\left\lbrace z \in \Bn : |1-\langle z,\zeta\rangle| <\frac{\gamma}{2}(1-|z|^2)\right\rbrace.
\]
Let $0<q,s<\infty$ and $\mu$ be a positive Borel measure on $\Bn$. The tent space $T^{q}_{s}(\mu)$ consists of those $\mu$-measurable functions $f:\Bn \rightarrow \mathbb{C}$ with
$$\|f\|_{T^{q}_{s}(\mu)}^q : =\int_{\Sn} \left(\int_{\Gamma(\zeta)} |f(z)|^s \frac{d\mu(z)}{(1-|z|^2)^n}\right)^{q/s}d\sigma(\zeta)<\infty.$$
The aperture $\gamma>1$ of the Kor\'anyi region is suppressed from the notation, as it is well-known that any two apertures generate the same tent space with equivalent quasinorms.

A full description of the boundedness of the embedding from Hardy spaces $H^p$ into the tent spaces $T^q_ s(\mu)$ in one dimension in the case $p\le q$ was obtained by Cohn \cite{Cohn} and Gong-Lou-Wu \cite{GLW} in terms of the area-type operator $A_{\mu,s}$ defined as
\[
A_{\mu,s} f(\zeta)= \left (\int_{\Gamma(\zeta)} |f(z)|^s \,\frac{d\mu(z)}{(1-|z|^2)^n} \right )^{1/s}.
\]
It is obvious that $I:H^p\rightarrow T^q_ s(\mu)$ is bounded if and only if $A_{\mu,s}: H^p \rightarrow L^q(\Sn)$ is bounded. Related results can be found in \cite{Am, C-O, PP1, WuIEOT, Wu06}.

 Our first main result is a generalization of the description of the boundedness of the area operator $A_{\mu,s}: H^p \rightarrow L^q(\Sn)$ in the case $p\le q$ to the setting of higher dimensions. We mention here, that the proof given in one dimension in \cite{Cohn, GLW} was somewhat technical, using tools of harmonic analysis such as Calderon-Zygmund decompositions among others. The proof we present here, valid for all dimensions, is much more simpler and natural, and it is based on a trick that has its roots on \cite{P1} when studying integration-type operators acting on Hardy spaces.
\begin{theorem}\label{MT1}
Let $0<p\le q <\infty$, and $\mu$ be a positive Borel measure on $\Bn$. Then $A_{\mu,s}:H^p \rightarrow L^q(\Sn)$ is bounded, if and only if, $\mu$ is a $\beta$-Carleson measure, with $\beta=1+s(\frac{1}{p}-\frac{1}{q})$. Moreover, we have
\[
\big \| A_{\mu,s} \big \|_{H^p\rightarrow L^q(\Sn)}\asymp \|\mu\|^{1/s}_{CM_{\beta}}.
\]
\end{theorem}
We recall that, for $\alpha>0$, a finite positive Borel measure on $\Bn$ is called an $\alpha$-Carleson measure if
$
\mu(B_{\delta}(\xi))\lesssim \delta^{n\alpha}
$ for all $\xi\in \Sn$ and $\delta>0$, where $B_{\delta}(\xi)$ denotes the non-isotropic metric ball
$$B_{\delta}(\xi)=\left\lbrace z\in \Bn: |1-\langle z, \xi \rangle |<\delta \right\rbrace.$$
The famous Carleson measure embedding theorem  was extended to several complex variables by H\"{o}rmander \cite{H}, and to the case $p< q$ by Duren \cite {Du}. A simple proof of Duren's theorem in the setting of the unit ball can be found in \cite{P1} for example. Joining all the results, the Carleson-H\"{o}rmander-Duren's theorem asserts that, for $0<p\le q<\infty$, the embedding $I:H^p\rightarrow L^q(\mu):=L^q(\Bn,\mu)$ is bounded if and only if $\mu$ is a $q/p$-Carleson measure. Moreover, one has the estimate $\|I\|_{H^p\rightarrow L^q(\mu)}\asymp \|\mu\|^{1/q}_{CM_{q/p}},$ where, for $\alpha>0$, we set
 $$\|\mu\|_{CM_{\alpha}}:=\sup_{\xi \in \Sn, \delta>0} \mu(B_{\delta}(\xi))\delta^{-n\alpha}.$$

Concerning the description of the boundedness of $A_{\mu,s}: H^p \rightarrow L^q(\Sn)$ when $0<q<p<\infty$, in one dimension several cases where solved in \cite{GLW}. It was proved that, when $q>s$, and $0<q<p<\infty$, then $A_{\mu,s}: H^p(\mathbb{B}_ 1) \rightarrow L^q(\mathbb{S}_ 1)$ is bounded if and only if the function $\widetilde{\mu}$ belongs to $L^{r/s} (\mathbb{S}_ 1)$, where
\[
\widetilde{\mu}(\zeta)=\int_{\Gamma(\zeta)} \frac{d\mu(z)}{(1-|z|^2)^n},\qquad \zeta \in \Sn.
\]
Really, in \cite{GLW}, this result was proved for $s=1$, but in one dimension once such a description is obtained for a fixed $s$, the result for general $s$ follows directly from the strong factorization for functions in Hardy spaces. In \cite{GLW}, they also conjecture that the description obtained holds also for all possible values of $p,q,s$. Our second main result is a proof of this conjecture in all dimensions.
\begin{theorem}\label{MT2}
 Let $0<q<p<\infty$ and $s>0$. Let $\mu$ be a positive Borel measure on $\Bn$. Then the operator $A_{\mu,s}: H^p\rightarrow L^q(\Sn)$ is bounded if and only if $\widetilde{\mu} \in L^{r/s}(\Sn)$ with $r=pq/(p-q)$. Moreover, one has
 \[
 \big \| A_{\mu,s} \big \|_{H^p \rightarrow L^q(\Sn)}\asymp \big \|\widetilde{\mu} \big \|_{L^{r/s}(\Sn)}^{1/s} .
 \]
\end{theorem}

The proof of the sufficiency of the condition $\widetilde{\mu}\in L^{r/s} (\Sn)$ is somewhat easy, so that the interest of Theorem \ref{MT2} is its necessity. The case $r>s$ is proved using Luecking description \cite{Lue1} of those positive Borel measures $\mu$ on $\Bn$ for which the embedding $I: H^p \rightarrow L^q(\mu)$ is bounded for $0<q<p<\infty$. That is, Luecking's theorem states that for $0<q<p<\infty$, the Carleson embedding $I: H^p \rightarrow L^q(\mu)$ is bounded if and only if $\widetilde{\mu}\in L^{\frac{p}{p-q}}(\Sn)$, and moreover one has the estimate $ \|I\|_{H^p\rightarrow L^q(\mu)} \asymp  \|\widetilde{\mu} \|^{1/q}_{L^{\frac{p}{p-q}} (\Sn)}$.

The necessity in the case $0<r\le s$ seems to be much more challenging, and our proof uses: Khintchine's type inequalities, that has been widely used in recent years in order to get the necessity part  when studying the boundedness with loss of certain operators; the theory of tent sequence spaces, with a crucial use of the factorization of such tent spaces; and several new techniques whose ideas came to us when studying Luecking's results \cite{Lue1} on embedding derivatives of Hardy spaces into Lebesgue spaces. We hope that these new techniques are going to be useful in the future in order to deal with related problems.\\

Throughout the paper, constants are
often given without computing their exact values, and the value of a constant $C$ may change
from one occurrence to the next. We also use the notation $a\lesssim b$ to indicate that there is a constant $C>0$
with $a\le C b$. Also, the notation $a\asymp b$ means that the two quantities are comparable.\\

The paper is organized as follows: in Section \ref{S-p} we recall some well known results that will be used in the proofs. Theorems \ref{MT1} is proved in Section \ref{s3}, and Theorem \ref{MT2} is proved in Section \ref{s4}, section that is divided in several subsections according to what case is proved: sufficiency, and necessity for $r>s$, $r=s$ and $r<s$, as each of these cases requires different techniques.

\section{Some preliminaries}\label{S-p}
In this section, we are going to collect some results and estimates needed for the proofs of the main theorems of the paper.
\subsection{Tent sequence spaces}
 A sequence of points $\{z_ j\}$ in $\Bn$ is said to be separated if there exists $\delta>0$ such that $\beta(z_ i,z_ j)\ge \delta$ for all $i$ and $j$ with $i\neq j$, where $\beta(z,w)$ denotes the Bergman metric on $\Bn$. We use the notation $D(a,r)=\{z\in \Bn : \beta(z, a)<r\}$ for  the Bergman metric ball of radius $r>0$ centered at a point $a\in \Bn$.

 By \cite[Theorem 2.23]{ZhuBn},
there exists a positive integer $N$ such that for any $0<r<1$ one can
find a sequence $\{a_k\}$ in $\Bn$ with the  properties:
\begin{itemize}
\item[(i)  ] $\Bn=\bigcup_{k}D(a_k,r)$;
\item[(ii) ] The sets $D(a_k,r/4)$ are mutually disjoint;
\item[(iii) ] Each point $z\in\Bn$ belongs to at most $N$ of the sets $D(a_k,4r)$.
\end{itemize}

\noindent Any sequence $\{a_k\}$ satisfying the above conditions is said to be
an $r$-\emph{lattice}
(in the Bergman metric). It is clear that any $r$-lattice is a separated sequence.\\

For $0<p,q<\infty$ and a fixed separated sequence $Z=\{z_ j\}\subset \Bn$, let $T^p_ q(Z)$ consist of those sequences $\lambda=\{\lambda_ j\}$ of complex numbers with
\[ \|\lambda \|_{T_ q^p(Z)}^p :=\int_{\Sn} \!\!\Big (\!\!\sum_{z_ j\in \Gamma(\zeta)} \!|\lambda_ j |^q \Big )^{p/q} d\sigma(\zeta)  <\infty.\]
The following result can be thought as the holomorphic analogue of Lemma 3 in Luecking's paper \cite{Lue1}. The current version can be found in \cite{Ars, Jev, P1}.
\begin{otherp}\label{TKL}
Let $Z=\{z_ j\}$ be a separated sequence in $\Bn$ and let $0<p<\infty$. If $b>n\max(1,2/p)$, then the operator $T_{Z}: T^p_ 2(Z)\rightarrow H^p$ defined by
\[
T_{Z}(\{\lambda_ j\})=\sum_ j \lambda_ j \,\frac{(1-|z_ j|^2)^{b}}{(1-\langle z, z_ j  \rangle )^b}
\]
 is bounded.
\end{otherp}

We will also need the following duality result for the tent spaces of sequences \cite{Ars, Jev, Lue1}.

\begin{otherth}\label{TTD1}
Let $1<p,q<\infty$ and $Z=\{a_k\}$ be a separated sequence in $\Bn$. Then the dual of $T^p_q(Z)$ is isomorphic to $T^{p'}_{q'}(Z)$ under the pairing
\[
\langle \lambda ,\mu \rangle_{T^2_2(Z)} =\sum_ k \lambda _ k \,\overline{\mu_ k} (1-|a_ k|^2)^n,\quad \lambda \in T^p_q(Z),\quad \mu \in T^{p'}_{q'}(Z).
\]
\end{otherth}

A crucial step for the proof of Theorem \ref{MT2} will be an appropriate use of the following result concerning factorization of tent sequence spaces, which can be found in \cite{MPPW19}.

\begin{otherth}\label{TTD2}
Let $0<p,q<\infty$ and $Z=\{a_k\}$ be an $r$-lattice. Let $0<p<p_1, p_2<\infty$,  $0<q<q_1,q_2<\infty$  satisfy
\[
\frac{1}{p}=\frac{1}{p_1}+\frac{1}{p_2},\qquad \frac{1}{q}=\frac{1}{q_1}+\frac{1}{q_2}.
\]
Then
\[
T^p_q(Z)=T^{p_1}_{q_1}(Z)\cdot T^{p_2}_{q_2}(Z).
\]
That is, if $\alpha\in T^{p_ 1}_{q_ 1}(Z)$ and $\beta \in T^{p_2}_{q_2}(Z)$, then $\alpha\cdot \beta \in T^p_q(Z)$ with $$\|\alpha \cdot \beta \|_{T^p_q(Z)}\lesssim \|\alpha\|_{T^{p_ 1}_{q_ 1}(Z)}\cdot \|\beta \|_{T^{p_2}_{q_2}(Z)};$$ and conversely, if $\lambda \in T^p_ q(Z)$, then there are sequences $\alpha \in T^{p_ 1}_{q_ 1}(Z)$ and $\beta \in T^{p_2}_{q_2}(Z)$ such that $\lambda=\alpha \cdot \beta$, and $\|\alpha\|_{T^{p_ 1}_{q_ 1}(Z)}\cdot \|\beta \|_{T^{p_2}_{q_2}(Z)}\lesssim \|\lambda \|_{T^p_q(Z)}$.
\end{otherth}

\subsection{Some estimates and the admissible maximal function}
For $z \in \Bn$,  set
$
I(z)=\{\zeta \in \Sn : z \in \Gamma(\zeta)\}.
$
Since $\sigma(I(z))\asymp (1-|z|^2)^n$,  Fubini's theorem yields the estimate:
\begin{equation}\label{EqG}
\int_{\Bn} \varphi(z)d\nu(z)\asymp \int_{\Sn} \left (\int_{\Gamma(\zeta)} \varphi(z) \frac{d\nu(z)}{(1-|z|^2)^{n}} \right )d\sigma(\zeta),
\end{equation}
where $\varphi$ is any positive measurable function and $\nu$ is a finite positive measure.\\

We recall that, for a continuous function $f:\Bn\to \mathbb{C}$, the admissible (non-tangential) maximal function $f^*$ is defined by
$$f^*(\xi)=\sup_{z \in \Gamma(\xi)}|f(z)|,\qquad \xi \in \Sn.$$
An important well known result \cite[Theorem 4.24]{ZhuBn} is the $L^p$-boundedness of the admissible maximal function: for $0<p<\infty$ and $f\in H^p$, one has $\|f^*\|_{L^p(\Sn)}\lesssim \|f\|_{H^p}$.

\subsection{Area function description of Hardy spaces}
Another function we need is the admissible area function $A f$ defined on $\Sn$ by
\[Af(\zeta)=\left ( \int_{\Gamma(\zeta)} |Rf(z)|^2 \,(1-|z|^2)^{1-n}dv(z)\right )^{1/2},\]
where $Rf$ denotes the radial derivative of $f$.
The following result \cite{AB,FS, P1} describing the functions in the Hardy space in terms of the admissible area function, is the version for the unit ball of $\Cn$ of the famous Calder\'{o}n area theorem \cite{Cal}.
\begin{otherth}\label{AreaT}
Let $0<p<\infty$ and $g\in H(\Bn)$. Then $g\in H^p$ if and only if $Ag\in L^p(\Sn)$. Moreover, if $g(0)=0$ then
\[ \|g\|_{H^p}\asymp \|Ag\|_{L^p(\Sn)}.\]
\end{otherth}
\section{Proof of Theorem \ref{MT1}} \label{s3}

\subsection{Sufficiency} Suppose that $\mu$ is a $\beta$-Carleson measure, with $\beta=1+s(\frac{1}{p}-\frac{1}{q})$.
Assume first that $0<q\le s$. As, for $\alpha=\frac{p(s-q)}{q}$, we have
\[
|f(z)|^s  \le f^*(\zeta) ^{\alpha} |f(z)|^{s-\alpha}, \qquad z\in \Gamma(\zeta),
\]
then
\[
A_{\mu,s} f(\zeta)^q \le (f^*(\zeta))^{\alpha q/s}\left (\int_{\Gamma(\zeta)} |f(z)|^{s-\alpha} \,\frac{d\mu(z)}{(1-|z|^2)^n}\right )^{q/s}.
\]
Hence, by H\"{o}lder's inequality with exponents $s/(s-q)$ and $s/q$, we get
\[
\begin{split}
\big \|A_{\mu,s} f \big \|^q_{L^q(\Sn)} &  \le \|f^* \|_{L^p(\Sn)}^{\frac{\alpha q}{s}} \left (\int_{\Sn} \int_{\Gamma(\zeta)} |f(z)|^{s-\alpha}  \,\frac{d\mu(z)}{(1-|z|^2)^n} \,d\sigma(\zeta) \right )^{q/s}
\\
&\asymp
\|f^* \|_{L^p(\Sn)}^{\frac{ \alpha q}{s}} \left (\int_{\Bn}  |f(z)|^{s-\alpha} \,d \mu(z) \right )^{q/s}.
\end{split}
\]
Since $\frac{(s-\alpha)}{p}=1+s(\frac{1}{p}-\frac{1}{q})=\beta$, by the $L^p$-boundedness of the admissible maximal function and the Carleson-H\"{o}rmander-Duren's theorem, we obtain
\[
\begin{split}
\big \|A_{\mu,s} f \big \|^q_{L^q(\Sn)} & \lesssim \|f \|_{H^p}^{\frac{\alpha q}{s}} \cdot \|f\|_{L^{s-\alpha}(\mu)}^{\frac{q(s-\alpha)}{s}}
\lesssim \|f \|_{H^p}^{\frac{\alpha q}{s}} \cdot \|I\|_{H^p\rightarrow L^{s-\alpha}(\mu)}^{\frac{q(s-\alpha)}{s}} \cdot \|f\|_{H^p}^{\frac{q(s-\alpha)}{s}}
\\
& \asymp \big \|\mu \big \|_{CM_{\frac{(s-\alpha)}{p}}} ^{q/s}\cdot \|f\|_{H^p}^{q}.
\end{split}
\]
Therefore, $A_{\mu,s}: H^p \rightarrow L^q(\Sn)$ is bounded with $$\big \|A_{\mu,s}\big \|_{H^p \rightarrow L^q(\Sn)} \lesssim \big \|\mu \big \|_{CM_{\beta}} ^{1/s}.$$

Next, we consider the case $q>s$. As $\|A_{\mu,s}f \|^s_{L^q(\Sn)}=\|(A_{\mu,s} f)^s \|_{L^{q/s}(\Sn)}$, by duality we must show that, for any function $\varphi \in L^{q/(q-s)}(\Sn)$ with $\varphi \ge 0$, we have
\begin{equation}\label{Eq-T1-A}
\int_{\Sn} (A_{\mu,s} f )^s\,\varphi \,d\sigma \le C\,\big \|\mu \big \|_{CM_{\beta}} \cdot \|\varphi \|_{L^{\frac{q}{q-s}}(\Sn)}\cdot \|f\|_{H^p}^s.
\end{equation}
Now, as $1-|z|^2 \asymp |1-\langle z,\zeta \rangle |$ for $z\in \Gamma(\zeta)$, we have
\[
\begin{split}
\int_{\Sn} (A_{\mu,s} f )^s\,\varphi \,d\sigma &= \int_{\Sn} \int_{\Gamma(\zeta)} |f(z)|^s \,\frac{d\mu(z)}{(1-|z|^2)^n} \,\varphi(\zeta) \,d\sigma(\zeta)
\\
&\lesssim  \int_{\Sn} \int_{\Bn} |f(z)|^s \,\frac{(1-|z|^2)^n \,d\mu(z)}{|1-\langle z,\zeta \rangle |^{2n}} \,\varphi(\zeta) \,d\sigma(\zeta)
\\
& = \int_{\Bn} |f(z)|^s \, P\varphi(z) \,d\mu(z),
\end{split}
\]
where
\[
P\varphi(z)= \int_{\Sn} \varphi (\zeta) \frac{(1-|z|^2)^n}{|1-\langle z,\zeta \rangle |^{2n}} \,d\sigma(\zeta)
\]
is the invariant Poisson transform of $\varphi$. Let
\[
t=\frac{p}{s}+\frac{(q-p)}{q}.
\]
 As $q>s$, then $t>1$ with conjugate exponent
\[
t'=\frac{pq+s(q-p)}{p(q-s)}.
\]
Hence, applying H\"{o}lder's inequality, we obtain
\[
\int_{\Bn} |f(z)|^s \, P\varphi(z) \,d\mu(z) \le \left ( \int_{\Bn} |f(z)|^{st}\,d\mu(z)\right )^{1/t} \left (\int_{\Bn} \big (P\varphi(z) \big )^{t'}\,d\mu(z) \right )^{1/t'}.
\]
As $st/p= 1+s (\frac{1}{p}-\frac{1}{q})=\beta$, by the Carleson-H\"{o}rmander-Duren's theorem we get
\[
\int_{\Bn} |f(z)|^{st}\,d\mu(z) \lesssim \big \|\mu \big \|_{CM_{\beta}} \cdot \|f\|_{H^p}^{st}.
\]
On the other hand, we have
\[
\frac{t'}{(q/(q-s))}= \frac{pq+s(q-p)}{pq}= \beta,
\]
and by the well known version for Poisson integrals of the Carleson-H\"{o}rmander-Duren's theorem, we also have
\[
\int_{\Bn} \big (P\varphi(z) \big )^{t'}\,d\mu(z) \lesssim \|\mu \big \|_{CM_{\beta}} \cdot \|\varphi \|^{t'}_{L^{q/(q-s)}(\Sn)}.
\]
All together, we have
\[
\int_{\Sn} (A_{\mu,s} f )^s\,\varphi \,d\sigma \lesssim \int_{\Bn} |f|^s \, P\varphi \,d\mu \lesssim \|\mu \big \|_{CM_{\beta}} \cdot \|\varphi \|_{L^{\frac{q}{q-s}}(\Sn)}\cdot \|f\|_{H^p}^s,
\]
that proves \eqref{Eq-T1-A}.
\subsection{Necessity}
Set $\beta=1+s(\frac{1}{p}-\frac{1}{q})$. By standard approximation arguments, it is enough to prove the estimate $$\big \|\mu \big \|^{1/s}_{CM_{\beta}} \lesssim \big \| A_{\mu,s} \big \|_{H^p \rightarrow L^q(\Sn)},$$ assuming that $\mu$ is already a $\beta$-Carleson measure. It is well known \cite[Theorem 45]{ZZ} that, for any $t>0$, we have
\[
\big \|\mu \big \|_{CM_{\beta}} \asymp \sup_{a\in \Bn} (1-|a|^2)^t \int_{\Bn} \frac{d\mu(z)}{|1-\langle z,a \rangle |^{n\beta +t}}.
\]
For each $a\in \Bn$, consider the function $f_ a$ defined as
\[
f_ a(z)=\frac{(1-|a|^2)^{t/s}}{(1-\langle z,a \rangle )^{\frac{n\beta+t}{s}}},\qquad z\in \Bn.
\]
Then, by \eqref{EqG},
\[
(1-|a|^2)^t \int_{\Bn} \frac{d\mu(z)}{|1-\langle z,a \rangle |^{n\beta +t}}=\int_{\Bn} |f_ a(z)|^s \,d\mu(z)\asymp \int_{\Sn} A_{\mu,s}f_ a(\zeta)^s \,d\sigma(\zeta),
\]
so that
\begin{equation}\label{ET1-B}
\big \|\mu \big \|_{CM_{\beta}} \asymp \sup_{a\in \Bn}\int_{\Sn} (A_{\mu,s}f_ a)^s \,d\sigma.
\end{equation}
Choose $0<\varepsilon<1$ so that $\varepsilon s <q$. Then, by H\"{o}lder's inequality with exponent $\frac{q}{\varepsilon s}$,
\[
\begin{split}
 \int_{\Sn} (A_{\mu,s}f_ a)^s \,d\sigma &= \int_{\Sn} (A_{\mu,s}f_ a)^{\varepsilon s}\, (A_{\mu,s}f_ a)^{(1-\varepsilon)s} \,d\sigma
\\
& \le \left (\int_{\Sn} (A_{\mu,s} f_ a)^q \,d\sigma  \right ) ^{\varepsilon s/q}  \left (\int_{\Sn} (A_{\mu,s} f_ a)^{\frac{q(1-\varepsilon)s}{(q-\varepsilon s)}} \,d\sigma  \right ) ^{(q-\varepsilon s)/q}.
\end{split}
\]
By the boundedness of $A_{\mu,s}$, we have
\[
\int_{\Sn} (A_{\mu,s} f_ a)^q \,d\sigma \le \big \| A_{\mu,s} \big \|_{H^p\rightarrow L^q(\Sn)} ^q \cdot \|f_ a \|_{H^p}^{q}.
\]
Set
\[
q_ 1= \frac{q(1-\varepsilon)s}{(q-\varepsilon s)}; \qquad p_ 1=\frac{pq(1-\varepsilon)s}{pq+(q-p)s-\varepsilon qs}.
\]
It is easy to check that $1+s (\frac{1}{p_ 1}-\frac{1}{q_ 1})=\beta$. Therefore, by the sufficiency already proved in the previous subsection, we have
\[
\int_{\Sn} (A_{\mu,s} f_ a)^{\frac{q(1-\varepsilon)s}{(q-\varepsilon s)}} \,d\sigma \lesssim \big \| \mu \big \|^{q_ 1/s}_{CM_{\beta}}\cdot \|f_ a \|^{q_ 1}_{H^{p_ 1}}.
\]
Putting these estimates together, we obtain
\[
\int_{\Sn} (A_{\mu,s}f_ a)^s \,d\sigma \lesssim \big \| A_{\mu,s} \big \|_{H^p\rightarrow L^q(\Sn)} ^{\varepsilon s} \cdot \|f_ a \|_{H^p}^{\varepsilon s}\cdot \big \| \mu \big \|^{1-\varepsilon}_{CM_{\beta}}\cdot \|f_ a \|^{(1-\varepsilon)s}_{H^{p_ 1}}.
\]
Bearing in mind \eqref{ET1-B}, we get
\[
\big \|\mu \big \|^{\varepsilon}_{CM_{\beta}} \lesssim \big \| A_{\mu,s} \big \|_{H^p\rightarrow L^q(\Sn)} ^{\varepsilon s} \cdot\sup_{a\in \Bn} \left (\|f_ a \|_{H^p}^{\varepsilon s}\cdot \|f_ a \|^{(1-\varepsilon)s}_{H^{p_ 1}}\right ).
\]
Finally, by the typical integral estimates (see \cite[Theorem 1.12]{ZhuBn}) we have
\[
\|f_ a \|_{H^p}^{\varepsilon s}\cdot \|f_ a \|^{(1-\varepsilon)s}_{H^{p_ 1}} \asymp  1,
\]
that gives the desired result.

\section{Proof of Theorem \ref{MT2}}\label{s4}
\subsection{Sufficiency}
For $f\in H^p$, we have
\[
\begin{split}
\big \|A_{\mu,s} f \big \|^q_{L^q(\Sn)}&=\int_{\Sn} \left ( \int_{\Gamma(\zeta)} |f(z)|^s \,\frac{d\mu(z)}{(1-|z|^2)^n} \right )^{q/s} d\sigma(\zeta)
\\
& \le \int_{\Sn} f^*(\zeta)^{q}\,\widetilde{\mu}(\zeta)^{q/s}\,d\sigma(\zeta).
\end{split}
\]
Then, applying H\"{o}lder's inequality with exponent $\frac{p}{q}$ (that has conjugate exponent $\frac{p}{p-q}$), and using the $L^p$-boundedness of the admissible maximal function, we obtain
\[
\big \|A_{\mu,s} f \big \|^q_{L^q(\Sn)} \le \|f^* \|_{L^p(\Sn)} ^q \cdot \big \|\widetilde{\mu} \big \|_{L^{r/s}(\Sn)}^{q/s}\lesssim \|f \|_{H^p} ^q \cdot \big \|\widetilde{\mu} \big \|_{L^{r/s}(\Sn)}^{q/s}.
\]
Hence $A_{\mu,s}: H^p \rightarrow L^q(\Sn)$ is bounded with $\big \|A_{\mu,s}\big\|_{H^p\rightarrow L^q(\Sn)}\lesssim \big \|\widetilde{\mu}\big \|^{1/s}_{L^{r/s}(\Sn)}.$

\subsection{Necessity: the case $r>s$}
 Suppose $A_{\mu, s}:H^p\rightarrow L^q(\Sn)$ is bounded. From Luecking's theorem, we have
\[
\big \|\widetilde{\mu} \big \|_{L^{r/s}(\Sn)} \asymp \sup_{\|f\|_{H^{p}}=1} \int_{\Bn} |f(z)|^t\,d\mu(z),
\]
where the number $t$ is chosen so that $\frac{p}{p-t}=\frac{r}{s}$, that is,
\[
t=\frac{(r-s)p}{r}.
\]
Observe that $t>0$ as $r>s$, and also
$p>t$. From \eqref{EqG}, we have
\[
\int_{\Bn} |f(z)|^t \,d\mu(z)  \asymp \int_{\Sn} \int_{\Gamma(\zeta)} |f(z)|^t \,\frac{d\mu(z)}{(1-|z|^2)^n} \,d\sigma(\zeta) .
\]
If $t=s$, then $q=s$ and we get directly the result. In case that $t>s$, then $q/s>1$ having conjugate exponent $q/(q-s)=p/(t-s)$. Then H\"{o}lder's inequality together with the $L^{p}$ boundedness of the admissible maximal function $f^*$ gives
\[
\big \|\widetilde{\mu} \big \|_{L^{r/s}(\Sn)} \lesssim  \sup_{\|f\|_{H^{p}}=1} \int_{\Sn} |f^*(\zeta)|^{t-s}\,\big (A_{\mu, s}f(\zeta)\big )^s \,d\sigma (\zeta) \lesssim \big \|A_{\mu, s} \big \|^s_{H^{p}\rightarrow L^{q}(\Sn)}.
\]
 It remains to deal with the case  $t<s$. By standard approximation arguments, it suffices to show the estimate $\|\widetilde{\mu}  \|_{L^{r/s}(\Sn)} \lesssim  \|A_{\mu, s} \|^s_{H^{p}\rightarrow L^{q}(\Sn)}$ assuming that the function $\widetilde{\mu}$ is already in $L^{r/s}(\Sn)$. By  H\"{o}lder's inequality  with exponent $s/t>1$,
we get
\[
 \int_{\Bn} |f(z)|^t \,d\mu(z) \lesssim \int_{\Sn} \big (A_{\mu, s}f(\zeta)\big )^t \,\widetilde{\mu}(\zeta)^{\frac{(s-t)}{s}}\,d\sigma(\zeta).
\]
It is also easy to check that $t<q$ since $q<s$. Then, using by H\"{o}lder's inequality again, now with exponent $q/t$ (that has conjugate exponent $q/(q-t)=r/(s-t)$), gives
\[
\int_{\Bn} |f(z)|^t \,d\mu(z) \lesssim  \big \| A_{\mu, s} f\big \|^{t}_{L^{q}(\Sn)} \cdot \big \|\widetilde{\mu} \big \|^{(s-t)/s}_{L^{r/s}(\Sn)} .
\]
From this, it follows that
\[
\big \|\widetilde{\mu} \big \|_{L^{r/s}(\Sn)}\asymp  \sup_{\|f\|_{H^{p}}=1} \int_{\Bn} |f(z)|^t\,d\mu(z) \lesssim \big \|A_{\mu, s} \big \|^{t}_{H^{p}\rightarrow L^{q}(\Sn)} \cdot \big \|\widetilde{\mu} \big \|^{(s-t)/s}_{L^{r/s}(\Sn)},
\]
and we obtain the desired result.

\subsection{Necessity: the case $r=s$} For the proof of this case, we use an averaging function related to the measure $\mu$. For $t>0$,
let
\[
 \widehat{\mu}_t(z)=\frac{\mu(D(z,2t))}{(1-|z|)^n},\qquad z\in \Bn.
\]
If  $spt (\mu)$ denotes the support of the measure $\mu$, it is then clear that $\widehat{\mu}_t(z)\neq 0$ for $z\in spt (\mu)$.

\begin{proposition}\label{pro-1}
Let $\mu$ a positive Borel measure on $\Bn$ and $1\le \sigma <\infty$. The following conditions are equivalent:
\begin{itemize}
\item[(i)] $\widetilde{\mu}\in L^1(\Sn)$;

\item[(ii)] There is a positive constant $C_{\mu}$ such that
\[
\int_{\Sn} \left (\int_{\Gamma(\zeta)\cap spt (\mu)}  \!\!\!\!|Rf(z)| \,\widehat{\mu}_ t (z)^{-1/2}\,(1-|z|^2)\,\frac{d\mu(z)}{(1-|z|^2)^n}\right )^{\frac{2\sigma}{2+\sigma}}\!\! \!\!\!d\sigma(\zeta)\le C_{\mu}\, \|f\|_{H^{\sigma}}^{\frac{2\sigma}{2+\sigma}}
\]
for all $f\in H^{\sigma}$.
\end{itemize}
 Moreover, one has
\[
C_{\mu}\asymp \|\widetilde{\mu}\|_{L^1(\Sn)}^{\sigma/(2+\sigma)}.
\]
\end{proposition}

\begin{proof}
(i) implies (ii). Observe first that
\begin{equation}\label{Eq-RS11}
\begin{split}
J_{\mu} f(\zeta):=&\int_{\Gamma(\zeta)\cap spt (\mu)} |Rf(z)|^2 \,\widehat{\mu}_ t (z)^{-1}\,(1-|z|^2)^2\,\frac{d\mu(z)}{(1-|z|^2)^n}
 \\
 & \lesssim \int_{\widetilde{\Gamma}(\zeta)} |Rf|^2 dv_{1-n},
\end{split}
\end{equation}
 where $\widetilde{\Gamma}(\zeta)$ is another approach region $\Gamma_{\gamma'}(\zeta)$ with a bigger aperture $\gamma'$ so that
\begin{equation}\label{setin}
\bigcup_{z \in \Gamma_\gamma(\zeta)}D(z,t)\subset \Gamma_{\gamma'}(\zeta).
\end{equation}
In order to prove \eqref{Eq-RS11}, we first use the subharmonic result in \cite[Lemma 2.24]{ZhuBn} to obtain
\[
|Rf(z)|^2 \lesssim \frac{1}{(1-|z|^2)^2}\int_{D(z,t)} |Rf(w)|^2 \, dv_{1-n} (w).
\]
This gives
\[
J_{\mu} f(\zeta) \lesssim \int_{\Gamma(\zeta)\cap spt (\mu)}  \!\!\left ( \int_{D(z,t)}|Rf|^2 dv_{1-n} \right ) \,\frac{d\mu(z)}{\mu(D(z,2t))}
\]
For $w\in D(z,t)$, we have $D(w,t)\subset D(z,2t)$, and $\mu(D(w,t))>0$ if $z\in spt (\mu)$. Therefore, if $A$ denotes the set of points $w\in \Bn$ with $\mu(D(w,t))>0$, then
\[
J_{\mu} f(\zeta) \lesssim \int_{\widetilde{\Gamma}(\zeta)\cap A} \, \frac{|Rf(w)|^2}{\mu(D(w,t))} \left ( \int_{D(w,t)}d\mu(z) \right ) dv_{1-n}(w)\lesssim \int_{\widetilde{\Gamma}(\zeta)} |Rf|^2 dv_{1-n},
\]
proving \eqref{Eq-RS11}.

Now, using Cauchy-Schwarz and the inequality in \eqref{Eq-RS11}, we see that
the quantity in (ii) is less than constant times
\[
\int_{\Sn} \left (\int_{\widetilde{\Gamma}(\zeta)} |Rf|^2 dv_{1-n} \right )^{\sigma/(2+\sigma)} \widetilde{\mu}(\zeta)^{\sigma/(2+\sigma)}\,d\sigma(\zeta).
\]
Finally, applying H\"{o}lder's inequality with exponents $(2+\sigma)/2$ and  $(2+\sigma)/\sigma$, together with the area description of functions in Hardy spaces, we see that the quantity in (ii) is less than constant times
\[
\big \|\widetilde{\mu}\big \|_{L^1(\Sn)}^{\sigma/(2+\sigma)}\, \left [\int_{\Sn} \left (\int_{\widetilde{\Gamma}(\zeta)} |Rf|^2 dv_{1-n} \right )^{\sigma/2} \!\!\!\!d\sigma(\zeta)\right ] ^{2/(2+\sigma)} \!\!\!  \lesssim \big \|\widetilde{\mu}\big \|_{L^1(\Sn)}^{\sigma/(2+\sigma)} \cdot \|f\|_{H^{\sigma}}^{\frac{2\sigma}{2+\sigma}}
\]
which shows that (i) implies (ii) with $C_{\mu} \lesssim \big \|\widetilde{\mu}\big \|_{L^1(\Sn)}^{\sigma/(2+\sigma)}$. \\

Now assume that (ii) holds. We test the inequality with the function
\begin{equation}\label{EqT-01}
F_ s (z)=\sum_ k r_ k(s)\,\lambda_ k \left ( \frac{1-|a_ k|^2}{1-\langle z,a_ k\rangle}\right )^{\gamma},
\end{equation}
where $Z=\{a_ k\}$ is a $t$-lattice, $r_ k(s)$ are Rademacher functions, $\gamma >n$ and $\lambda =\{\lambda_ k\} \in T^{\sigma}_ 2 (Z)$. By Proposition \ref{TKL} we know that $F_ s\in H^{\sigma}$ with $\|F_ s\|_{H^{\sigma}}\lesssim \|\lambda \|_{T^{\sigma}_ 2}$. Using the argument with Kahane and Khintchine's inequalities as in \cite{MPPW19}, we obtain
\[
\int_{\Sn} \left ( \sum_{a_ k\in\widetilde{\Gamma}(\zeta)} |\lambda_ k | \, \int_{D(a_ k,t)\cap spt (\mu)}\widehat{\mu}_{t}(z)^{-1/2}\,\frac{d\mu(z)}{(1-|z|^2)^n} \right )^{\frac{2\sigma}{2+\sigma}}\!\!\! \!\!\! \!\!\! d\sigma(\zeta) \lesssim C_{\mu} \,\|\lambda\|_{T^{\sigma}_ 2}^{2\sigma/(2+\sigma)}.
\]
Since, for $z\in D(a_ k,t)$, we have $\widehat{\mu}_ t(z)\lesssim \widehat{\mu}_{2t}(a_ k)$, we get
\begin{equation}\label{EqT1}
\int_{\Sn} \left ( \sum_{a_ k\in\widetilde{\Gamma}(\zeta) \cap A} |\lambda_ k | \, \widehat{\mu}_{2t}(a_ k)^{-1/2}\,\widehat{\mu}_ t(a_ k) \right )^{\frac{2\sigma}{2+\sigma}} \!\!\! \!\!\!d\sigma(\zeta) \lesssim C_{\mu} \,\|\lambda\|_{T^{\sigma}_ 2}^{2\sigma/(2+\sigma)}
,
\end{equation}
where $A=\{w\in \Bn: \widehat{\mu}_ t(w)>0\}$.
It is clear that $\widetilde{\mu}\in L^1(\Sn)$ if and only if $\mu$ is a finite measure with $\|\widetilde{\mu}\|_{L^1(\Sn)}\asymp \mu(\Bn)$. Then, in order to prove that $\widetilde{\mu}\in L^1(\Sn)$, it suffices to show that $\{\widehat{\mu}_ t(a_ k)\}\in T_ 1^1(Z)$, and this is equivalent to $\{\widehat{\mu}_ t(a_ k)^{1/4}\}\in T_ 4^4(Z)$. For proving that, we claim that it is enough to show that
\begin{equation}\label{Eq-T44}
 \Big \{\widehat{\mu}_{2t}(a_ k)^{-1/4}\,\widehat{\mu}_ t(a_ k)^{1/2} \Big \}_{k} \in T^4_ 4 (Z),
\end{equation}
where one takes the indexes $k$ with $\widehat{\mu}_ t(a_ k)\neq 0$. That is, it is enough to prove
\[
\sum_ {k:a_ k \in A} \widehat{\mu}_ t (a_ k)^2 \, \widehat{\mu}_ {2t} (a_ k)^{-1}(1-|a_ k|^2)^n  \lesssim C_{\mu}^{\frac{2+\sigma}{\sigma}}.
\]
Indeed, we have
\[
\Big \|\{\widehat{\mu}_ t(a_ k)^{1/4}\}\Big \|^4 _{T^4_ 4} \asymp  \sum_ {k:a_ k \in A} \widehat{\mu}_ t (a_ k) \, (1-|a_ k|^2)^n = \sum_ {k:a_ k \in A} \mu (D(a_ k,2t)) \asymp \|\widetilde{\mu}\|_{L^1(\Sn)}.
\]
Then, by the Cauchy-Schwarz inequality
\[
\begin{split}
\|\widetilde{\mu}\|_{L^1(\Sn)} &  \lesssim \left (\sum_ {k:a_ k \in A} \!\! \widehat{\mu}_ t (a_ k)^2 \, \widehat{\mu}_ {2t} (a_ k)^{-1}(1-|a_ k|^2)^n \right )^{1/2} \!\!\!
\left (\sum_ {k:a_ k \in A} \!\! \widehat{\mu}_ {2t} (a_ k) \, (1-|a_ k|^2)^n \right )^{1/2}
\\
&\lesssim \|\widetilde{\mu}\|_{L^1(\Sn)}^{1/2}\, \left (\sum_ {k:a_ k \in A} \widehat{\mu}_ t (a_ k)^2 \, \widehat{\mu}_ {2t} (a_ k)^{-1}(1-|a_ k|^2)^n \right )^{1/2}.
\end{split}
\]
Hence,
\[
\|\widetilde{\mu}\|_{L^1(\Sn)}\lesssim \sum_ {k:a_ k \in A} \widehat{\mu}_ t (a_ k)^2 \, \widehat{\mu}_ {2t} (a_ k)^{-1}(1-|a_ k|^2)^n.
\]
To prove \eqref{Eq-T44}, by the duality of tent sequence spaces in Theorem \ref{TTD1}, we must prove that
\[
\sum_ {k: a_ k \in A} |\alpha_ k| \,\widehat{\mu}_{2t}(a_ k)^{-1/4}\,\widehat{\mu}_ t(a_ k)^{1/2} \,(1-|a_ k|^2)^n \lesssim C_{\mu} ^{\frac{2+\sigma}{4\sigma}}\,\|\alpha \|_{T_{4/3}^{4/3}}
\]
for each $\alpha=\{\alpha_k\}\in T_{4/3}^{4/3}(Z)$. By \eqref{EqG} we have
\[
\begin{split}
\sum_ {k: a_ k\in A} |\alpha_ k| &\,\widehat{\mu}_{2t}(a_ k)^{-1/4}  \,\widehat{\mu}_ t(a_ k)^{1/2} \,(1-|a_ k|^2)^n
\\
& \asymp
\int_{\Sn} \left (\sum_ {a_ k\in \Gamma (\zeta)\cap A} |\alpha_ k| \,\widehat{\mu}_{2t}(a_ k)^{-1/4}\,\widehat{\mu}_ t(a_ k)^{1/2}\right ) d\sigma(\zeta).
\end{split}
\]
Using the factorization of tent sequence spaces in Theorem \ref{TTD2}, we can write
\[
\alpha_ k =\beta_ k \cdot \lambda_ k ^{1/2},\qquad \lambda \in T_ 2^{\sigma}(Z),\quad  \beta \in T_ 2^{4\sigma/(3\sigma-2)}(Z).
\]
with $\|\beta \|_{T_ 2^{4\sigma/(3\sigma-2)}}\cdot \|\lambda \|^{1/2}_{T_ 2^{\sigma}}\lesssim \|\alpha\|_{T_{4/3}^{4/3}}.$ Then
\[
\begin{split}
\sum_ {a_ k\in \Gamma (\zeta)\cap A} &|\alpha_ k| \,\widehat{\mu}_{2t}(a_ k)^{-1/4}\,\widehat{\mu}_ t(a_ k)^{1/2}
\\
&\le \Big (\sum_ {a_ k\in \Gamma (\zeta)}  |\beta_ k|^2 \Big  )^{1/2}\left (\sum_ {a_ k\in \Gamma (\zeta)\cap A}  |\lambda_ k| \, \widehat{\mu}_{2t}(a_ k)^{-1/2}\,\widehat{\mu}_ t(a_ k) \right )^{1/2}.
\end{split}
\]
Finally, an application of H\"{o}lder's inequality with exponent $4\sigma/(2+\sigma)$ that has conjugate exponent $4\sigma/(3\sigma-2)$, together with our condition \eqref{EqT1}  gives
\[
\begin{split}
\sum_ {k: a_ k \in A} |\alpha_ k| & \,\widehat{\mu}_{2t}(a_ k)^{-1/4}\,\widehat{\mu}_ t(a_ k)^{1/2} \,(1-|a_ k|^2)^n
\\
& \lesssim \|\beta \|_{T_ 2^{4\sigma/(3\sigma-2)}} \, C_{\mu}^{\frac{2+\sigma}{4\sigma}} \, \|\lambda \|^{1/2}_{T_ 2^{\sigma}} \lesssim C_{\mu}^{\frac{2+\sigma}{4\sigma}}\, \|\alpha\|_{T_{4/3}^{4/3}}.
\end{split}
\]
This proves that $\widetilde{\mu} \in L^1(\Sn)$ with $ \|\widetilde{\mu}\|_{L^1(\Sn)}^{\sigma/(2+\sigma)} \lesssim C_{\mu}$, finishing the proof of the Proposition.
\end{proof}

As a consequence of the previous proposition, we get the following result that will be the key for the proof of the necessity of Theorem \ref{MT2} when $r=s$.

\begin{coro}\label{C1}
Let $\mu$ be a positive Borel measure on $\Bn$. We have that $\widetilde{\mu}\in L^1(\Sn)$ if and only if, for $0<p<\infty$, we have
\[
\int_{\Bn\cap spt (\mu)} |g(z)|^{p/4} |Rf(z)| \,\widehat{\mu}_ t (z)^{-1/2}\,(1-|z|^2)\,d\mu(z) \le K_{\mu}\cdot \|g\|_{H^p}^{p/4}\cdot \|f\|_{H^4}.
\]
Moreover, one has
\[
K_{\mu}\asymp \|\widetilde{\mu}\|_{L^1(\Sn)}^{1/2}.
\]
\end{coro}
\begin{proof}
Set $d\nu_ f(z)=|Rf(z)| \,\widehat{\mu}_ t (z)^{-1/2}\,(1-|z|^2)\,\chi_{spt(\mu)}(z)\,d\mu(z)$. Then the inequality is
\[
\int_{\Bn}|g(z)|^{p/4}d\nu_ f(z) \le K_{\mu}\cdot  C_ f \,\|g\|_{H^p}^{p/4},
\]
with $C_ f \asymp \|f\|_{H^4}$. Now, by  Luecking's theorem we have that $\widetilde{\nu_ f} \in L^{4/3}(\Sn)$ with
\[
\|\widetilde{\nu_ f}\|_{L^{4/3}(\Sn)}\asymp K_{\mu}\, \|f\|_{H^4},
\] and the result is a consequence of Proposition \ref{pro-1} with $\sigma=4$.
\end{proof}
\mbox{}
\\
Now we are ready for the proof of the necessity in Theorem \ref{MT2} when $r=s$.

\begin{theorem}\label{mt2-c}
 Let $\mu$ be a positive Borel measure on $\Bn$, and  $0<q<p<\infty$,  $s>0$,  with $r=pq/(p-q)=s$. If $A_{\mu, s}:H^p\rightarrow L^q(\Sn)$ is bounded, then $\widetilde{\mu} \in L^{1}(\Sn)$. Moreover, one has
 \[
 \|\widetilde{\mu}\|_{L^1(\Sn)} \lesssim \big \|A_{\mu,s}\big \|^s_{H^p \rightarrow L^q(\Sn)}.
 \]
\end{theorem}

\begin{proof}
We are going to use Corollary \ref{C1}. By Cauchy-Schwarz inequality, the estimate in \eqref{Eq-RS11}, H\"{o}lder's inequality with exponents $4$ and $4/3$, and the area functions description of Hardy spaces, we have
\[
\begin{split}
\int_{\Bn \cap spt (\mu)}& |g(z)|^{p/4} \,|Rf(z)| \,\widehat{\mu}_ t (z)^{-1/2}\,(1-|z|^2)\,d\mu(z)
\\
& \lesssim \int_{\Sn} \! \left ( \int_{\Gamma(\zeta)} \! \! |g(z)|^{p/2}\frac{d\mu(z)}{(1-|z|^2)^n}\right )^{1/2} \!\! \left ( \int_{\widetilde{\Gamma}(\zeta)} \!\! |Rf(z)|^{2}\,dv_{1-n}(z)\right )^{1/2} \!\!\! \! \! d\sigma(\zeta)
\\
& \lesssim \|f\|_{H^4} \left \{\int_{\Sn} \left (\int_{\Gamma(\zeta)} |g(z)|^{p/2} \frac{d\mu(z)}{(1-|z|^2)^n}\right )^{2/3} d\sigma(\zeta)\right \}^{3/4}.
\end{split}
\]
Hence, we need to estimate
\[
K_ g(\mu):=\int_{\Sn} \left (\int_{\Gamma(\zeta)} |g(z)|^{p/2} \frac{d\mu(z)}{(1-|z|^2)^n}\right )^{2/3} d\sigma(\zeta).
\]
If $p=2s$, as $r=s$, we see that $q=2s/3$, so that, in this case, we have
\[
K_ g(\mu)=\Big \| A_{\mu,s} g \Big \|_{L^q(\Sn)}^q,
\]
and the result follows by the boundedness of $A_{\mu,s}: H^p\rightarrow L^q(\Sn)$ and Corollary \ref{C1}.\\
\mbox{}
\\
If $p>2s$, then
\[
\begin{split}
K_ g(\mu) & \le \int_{\Sn} |g^*(\zeta)|^{(p-2s)/3}\,A_{\mu,s} g(\zeta)^{2s/3}\,d\sigma(\zeta)
\\
& \le \|g^*\|_{L^p(\Sn)}^{(p-2s)/3}\cdot \Big \| A_{\mu,s} g \Big \|_{L^{ps/(p+s)}(\Sn)}^{2s/3}.
\end{split}
\]
Since $pq/(p-q)=s$, we see that $ps=q(p+s)$. Then, by the $L^p$-boundedness of the admissible maximal function, we get
\[
K_ g(\mu)  \lesssim \|g\|^p_{H^p} \cdot \big \| A_{\mu,s} \big \|_{H^p \rightarrow L^q(\Sn)} ^{2s/3}
\]
 and we get the result from Corollary \ref{C1}.\\
\mbox{}
\\
If $p<2s$, then we apply H\"{o}lder's inequality with exponent $2s/p>1$ to obtain
\[
\begin{split}
K_ g(\mu) & \le \int_{\Sn} A_{\mu,s} g(\zeta)^{p/3}\,\widetilde{\mu}(\zeta)^{(2s-p)/3s} d\sigma(\zeta)
\\
& \le \big \|\widetilde{\mu} \big \|_{L^1(\Sn)}^{(2s-p)/3s}\cdot \Big \| A_{\mu,s} g \Big \|_{L^{ps/(p+s)}(\Sn)}^{p/3}
 \\
 & \le \big \|\widetilde{\mu} \big \|_{L^1(\Sn)}^{(2s-p)/3s}\cdot \Big \| A_{\mu,s} \Big \|^{p/3}_{H^p\rightarrow L^q(\Sn)} \cdot \| g \|_{H^p}^{p/3}.
\end{split}
\]
Hence, assuming that $\mu$ is already in $L^1(\Sn)$, by Corollary \ref{C1} we have
\[
\big \|\widetilde{\mu} \big \|^{1/2}_{L^1(\Sn)}\lesssim  \big \|\widetilde{\mu} \big \|_{L^1(\Sn)}^{(2s-p)/4s}\cdot \Big \| A_{\mu,s} \Big \|^{p/4}_{H^p\rightarrow L^q(\Sn)}
\]
that implies the estimate $\big \|\widetilde{\mu} \big \|_{L^1(\Sn)}\lesssim  \big \| A_{\mu,s} \big \|^{s}_{H^p\rightarrow L^q(\Sn)}$. The general case follows from an standard approximation argument.
\end{proof}
\mbox{}
\\
\subsection{Necessity: the case $r<s$}
We need first the following lemma.

\begin{lemma}\label{L2}
 Let $0<p, q <\infty$, $s>0$. Let $\mu$ be a positive Borel measure on $\Bn$. If $A_{\mu,s}:H^p \rightarrow L^q(\Sn)$ is bounded, then  $A_{\mu,ms}: H^{mp}\rightarrow L^{mq}(\Sn)$ is bounded for each positive integer $m$. Moreover, one has
\[
\big \|A_{\mu,ms}\big \|_{H^{mp}\rightarrow L^{mq}(\Sn)} \le \big \|A_{\mu,s}\big \|_{H^{p}\rightarrow L^{q}(\Sn)}^{1/m}.
\]
\end{lemma}

\begin{proof}
 If $f\in H^{mp}$, then  $f^m \in H^p$ with $\|f^m\|_{H^p}=\|f\|_{H^{mp}}^{m}$ and
\[
A_{\mu,ms}f(\zeta)=\left[A_{\mu,s}(f^m)(\zeta)\right]^{1/m}.
\]
Therefore,
\[
\begin{split}
\big \|A_{\mu,ms}f \big \|^{mq}_{L^{mq}(\Sn)} & \leq \big \|A_{\mu,s}(f^m) \big \|^{q}_{L^q(\Sn)}
\\
& \le \big \|A_{\mu,s} \big \|_{H^p \rightarrow L^q(\Sn)} ^q  \|f^m\|^q_{H^p}=\big \|A_{\mu,s} \big \|_{H^p \rightarrow L^q(\Sn)} ^q \,\|f\|_{H^{mp}}^{mq},
\end{split}
\]
which proves the desired estimate.
\end{proof}

The next result needed can be considered as the analogue of Proposition \ref{pro-1} and Corollary \ref{C1} for the case $r<s$.

\begin{proposition}\label{L1-A} Suppose $0<q<p<\infty$, $s>0$, $r/s<1$ with $r=pq/(p-q)$. For any positive integer $m$ with
\begin{equation}\label{EqT-10}
 mq>2 \   \textrm{ and }  \   m(p-q)s>2q,
 \end{equation} set
 \begin{equation}\label{EqT-04}
\sigma=\frac{2m(sp-sq-pq)}{m s(p-q)-2q}.
\end{equation}
Let $\mu$ be a positive Borel measure on $\Bn$. The following conditions are equivalent:
\begin{itemize}
\item [(i)] $\widetilde{\mu} \in L^{r/s}(\Sn)$;
\item [(ii)] There is a positive constant $K$ such that
\[
\int_{\Bn \cap spt(\mu)} \!\!\!\! |g(z)|^{\sigma} \,|Rf(z)|^{\sigma} \, \widehat{\mu}_{t}(z)^{-\sigma/2}\,(1-|z|^2)^{\sigma}\,d\mu(z) \le K \,\|g\|^{\sigma}_{H^{2mp}}\,\|f\|^{\sigma}_{H^{2mp}}.
\]
\end{itemize}
Moreover, we have
\[
\|\widetilde{\mu}\|^{1-\sigma/2}_{L^{r/s}(\Sn)}\asymp \sup \int_{\Bn \cap spt(\mu)} |g(z)|^{\sigma} \,|Rf(z)|^{\sigma} \, \widehat{\mu}_{t}(z)^{-\sigma/2}\,(1-|z|^2)^{\sigma}\,d\mu(z),
\]
where the supremum is taken over all functions $f,g\in H^{2mp}$ with $\|f\|_{H^{2mp}}=\|g\|_{H^{2mp}}=1$.
\end{proposition}

\begin{proof} First, suppose $\widetilde{\mu}\in L^{r/s}(\Sn)$, and let $f,g\in H^{2mp}$. Proceeding in the same way as in the proof of the estimate \eqref{Eq-RS11}, we have
\[
\int_{\Gamma(\zeta)\cap spt (\mu)} \!\! \! \! |g(z)|^{2}  |Rf(z)|^2 \,\widehat{\mu}_ t (z)^{-1}\,(1-|z|^2)^2\,\frac{d\mu(z)}{(1-|z|^2)^n} \lesssim \int_{\widetilde{\Gamma}(\zeta)} |g|^{2} |Rf|^2 dv_{1-n}.
\]
Observe that, as $r/s<1$, then $s(p-q)-pq>0$. This, together with \eqref{EqT-10} tells us that  $\sigma>0$. Also, as $mp>2$, it follows that $\sigma<2$. Now, an application of the estimate \eqref{EqG} together with  H\"{o}lder's inequality with exponent $2/\sigma>1$ gives
\[
\begin{split}
 &\int_{\Bn \cap spt(\mu)}\!\!\! |g(z)|^{\sigma} \,|Rf(z)|^{\sigma} \, \widehat{\mu}_{t}(z)^{-\sigma/2}\,(1-|z|^2)^{\sigma}\,d\mu(z)
\\
& \lesssim\int_{\Sn} \left(\int_{\Gamma(\zeta)\cap spt (\mu)} \!\!\!\! |g(z)|^{2} \,|Rf(z)|^{2} \, \widehat{\mu}_{t}(z)^{-1}\,\frac{(1-|z|^2)^2 d\mu(z)}{(1-|z|^2)^n}\right)^{\sigma/2} \!\!\! \widetilde{\mu}(\zeta)^{\frac{2-\sigma}{2}}d\sigma(\zeta)
\\
&\lesssim \int_{\Sn} \left(\int_{\Gamma(\zeta)}|g(z)|^{2} \,|Rf(z)|^{2} \,dv_{1-n}(z)\right)^{\sigma/2}\widetilde{\mu}(\zeta)^{\frac{2-\sigma}{2}}d\sigma(\zeta).
\end{split}
\]
As $0<\sigma<2$, the condition $r<s$ implies that $\frac{2r}{s(2-\sigma)}>1$ with conjugate exponent given by
\[
\frac{2r}{2r-s(2-\sigma)}=\frac{mp}{\sigma}.
\]
Hence, by H\"{o}lder's inequality with exponent $\frac{2r}{s(2-\sigma)}$, we obtain
\[
\begin{split}
 &\int_{\Bn \cap spt(\mu)} |g(z)|^{\sigma} \,|Rf(z)|^{\sigma} \, \widehat{\mu}_{t}(z)^{-\sigma/2}\,(1-|z|^2)^{\sigma}\,d\mu(z)
\\
& \lesssim\left(\int_{\Sn} \left(\int_{\Gamma(\zeta)}|g(z)|^{2} \,|Rf(z)|^{2} \,dv_{1-n}(z)\right)^{\frac{mp}{2}} \! \! \! d\sigma(\zeta)\right)^{\frac{\sigma}{mp}}\left(\int_{\Sn}\widetilde{\mu}(\zeta)^{r/s}d\sigma(\zeta)\right)^{\frac{s(2-\sigma)}{2r}}.
\end{split}
\]
By Cauchy-Schwarz inequality, the $L^p$-boundedness of the admissible maximal function and the area description of functions in Hardy spaces, we have
\[
\begin{split}
 \int_{\Sn}& \left(\int_{\Gamma(\zeta)}|g(z)|^{2} \,|Rf(z)|^{2} \,dv_{1-n}(z)\right)^{\frac{mp}{2}}d\sigma(\zeta)
\\
& \leq\int_{\Sn} g^{*}(\zeta)^{mp}\left(\int_{\Gamma(\zeta)} \,|Rf(z)|^{2} \,dv_{1-n}(z)\right)^{\frac{mp}{2}}d\sigma(\zeta)
\\
& \lesssim\|g\|^{mp}_{H^{2mp}}\cdot \|f\|^{mp}_{H^{2mp}}.
\end{split}
\]
Thus,
\[
\begin{split}
\sup  \int_{\Bn \cap spt (\mu)} \! \! |g(z)|^{\sigma} \,|Rf(z)|^{\sigma} \, \widehat{\mu}_{t}(z)^{-\sigma/2}\,(1-|z|^2)^{\sigma}\,d\mu(z) \lesssim \big \|\widetilde{\mu}\big \|^{\frac{2-\sigma}{2}}_{L^{r/s}(\Sn)},
\end{split}
\]
where the supremum is taken over all functions $f,g\in H^{2mp}$ with $\|f\|_{H^{2mp}}=\|g\|_{H^{2mp}}=1$.\\
\mbox{}
\\
For the converse, we need to show
\begin{equation}\label{iEqT-10}
\|\widetilde{\mu}\|^{1-\sigma/2}_{L^{r/s}(\Sn)}\lesssim  \int_{\Bn} |g(z)|^{\sigma} \,|Rf(z)|^{\sigma} \, \widehat{\mu}_{t}(z)^{-\sigma/2}\,(1-|z|^2)^{\sigma}\,d\mu(z)
\end{equation}
for  all functions $f,g\in H^{2mp}$ with $\|f\|_{H^{2mp}}=\|g\|_{H^{2mp}}=1$. For each $f,g\in H^{2mp}$, set
\[
d\nu_ f(z)=|Rf(z)|^{\sigma} \, \widehat{\mu}_{t}(z)^{-\sigma/2}\,(1-|z|^2)^{\sigma}\,\chi_{spt (\mu)}(z)\,d\mu(z).
\]
Then, by (ii) we have
\[
\int_{\Bn} |g(z)|^{\sigma} \,d\nu_ f(z) \le C_ f \,\|g\|_{H^{2mp}}^{\sigma},
\]
with $C_ f \asymp \|f\|^{\sigma}_{H^{2mp}}$. By Luecking's theorem, we know that this is equivalent to the condition $\widetilde{\nu_ f} \in L^{\frac{2mp}{2mp-\sigma}}(\Sn)$, with
\[
\|\widetilde{\nu_ f}\|_{L^{\frac{2mp}{2mp-\sigma}}(\Sn)} \asymp C_ f \asymp \|f\|^{\sigma}_{H^{2mp}}.
\]
That is, we have
\[
\int_{\Sn} \left ( \int_{\Gamma(\zeta)\cap spt(\mu)} \! \! \!\!\! \!\!   |Rf(z)|^{\sigma} \, \widehat{\mu}_{t}(z)^{-\sigma/2}\,(1-|z|^2)^{\sigma}\,\frac{d\mu(z)}{(1-|z|^2)^n} \right )^{\frac{2mp}{2mp-\sigma}} \! \! \! \!\!\!\!\!  d\sigma(\zeta) \lesssim \,\big \|f \big \|_{H^{2mp}}^{\frac{2mp\sigma}{2mp-\sigma}}.
\]
 Now we proceed as in the proof of Proposition \ref{pro-1}. We only sketch the proof without giving all the details. We test this inequality with the function $F_ s$ given by \eqref{EqT-01},
where $\lambda =\{\lambda_ k\} \in T^{2mp}_ 2 (Z)$. We know that $F_ s\in H^{2mp}$ with $\|F_ s\|_{H^{2mp}}\lesssim \|\lambda \|_{T^{2mp}_ 2}$. Using the argument with Kahane and Khintchine's inequalities, we obtain
%\[
%\int_{\Sn} \left ( \sum_{a_ k\in\widetilde{\Gamma}(\zeta)} |\lambda_ k |^{\sigma}\,(1-|a_ k|)^{-\sigma} \, \int_{D(a_ %k,r)}\widehat{\mu}_{t}(z)^{-\sigma/2}\,\frac{(1-|z|^2)^{\sigma}\,d\mu(z)}{(1-|z|^2)^n} \right )^{\frac{2mp}{2mp-\sigma}}\,d\sigma(\zeta) \le C %\|\lambda\|_{T^{2mp}_ 2}^{\frac{2mp}{2mp-\sigma}}.
%\]
%Since, for $z\in D(a_ k,r)$, we have $\widehat{\mu}_ t(z)\lesssim \widehat{\mu}_{2t}(a_ k)$, we have
\begin{equation}\label{EqT-02}
\int_{\Sn} \left ( \sum_{a_ k\in\widetilde{\Gamma}(\zeta)} |\lambda_ k |^{\sigma} \, \widehat{\mu}_{2t}(a_ k)^{-\sigma/2}\,\widehat{\mu}_ t(a_ k) \right )^{\frac{2mp}{2mp-\sigma}} \!\!\!\!d\sigma(\zeta) \le C \,\big \|\lambda \big \|_{T^{2mp}_ 2}^{\frac{2mp}{2mp-\sigma}}.
\end{equation}
In order to prove that $\widetilde{\mu}\in L^{r/s}(\Sn)$, it suffices to show that $\{\widehat{\mu}_ t(a_ k)\}\in T_ 1^{r/s}(Z)$, and this is equivalent to $\{\widehat{\mu}_ t(a_ k)^{\frac{2-\sigma}{2}}\}\in T_ {\frac{2}{2-\sigma}}^{\frac{2r}{s(2-\sigma)}}(Z)$. For proving that, we know that it is enough to show that
$
 \big \{\widehat{\mu}_{2t}(a_ k)^{-\sigma/4}\,\widehat{\mu}_ t(a_ k)^{1/2} \big \}_k \in T_ {\frac{4}{2-\sigma}}^{\frac{4r}{s(2-\sigma)}}(Z).
$
By duality, we must prove that
\[
\sum_ k |\alpha_ k| \,\widehat{\mu}_{2t}(a_ k)^{-\sigma/4}\,\widehat{\mu}_ t(a_ k)^{1/2}  \,(1-|a_ k|^2)^n \lesssim \|\alpha \|_{T^{\frac{4r}{4r-2s+\sigma s}}_{\frac{4}{2+\sigma}}}
\]
for each $\alpha=\{\alpha_k\}\in T^{\frac{4r}{4r-2s+\sigma s}}_{\frac{4}{2+\sigma}}(Z)$.
We have
\[
\begin{split}
\sum_ k |\alpha_ k| \,& \widehat{\mu}_{2t}(a_ k)^{-\sigma/4}\,\widehat{\mu}_ t(a_ k)^{1/2} \,(1-|a_ k|^2)^n
\\
&\asymp
\int_{\Sn} \left (\sum_ {a_ k\in \Gamma (\zeta)} |\alpha_ k| \,\widehat{\mu}_{2t}(a_ k)^{-\sigma/4}\,\widehat{\mu}_ t(a_ k)^{1/2}\right ) d\sigma(\zeta).
\end{split}
\]
Since, from \eqref{EqT-04}, 
 $$
T^{\frac{4r}{4r-2s+\sigma s}}_{\frac{4}{2+\sigma}}(Z)=T_{2}^{\frac{4mp}{2mp+\sigma}}(Z)\cdot T_{4/\sigma}^{\frac{4mp}{\sigma}}(Z)
$$
 using the factorization of tent sequence spaces we can write
$
\alpha_ k =\beta_ k \cdot \lambda_ k ^{\sigma/2}$, with  $\lambda \in T_ 2^{2mp}(Z)$ and $\beta \in T_ 2^{\frac{4mp}{2mp+\sigma}}(Z)$
with $\|\beta \|_{T_ 2^{\frac{4mp}{2mp+\sigma}}}\cdot \|\lambda \|^{\sigma/2}_{T_ 2^{2mp}}\lesssim \|\alpha\|_{T^{\frac{4r}{4r-2s+\sigma s}}_{\frac{4}{2+\sigma}}}.$ Then
\[
\begin{split}
\sum_ {a_ k\in \Gamma (\zeta)} &|\alpha_ k| \,\widehat{\mu}_{2t}(a_ k)^{-\sigma/4}\,\widehat{\mu}_ t(a_ k)^{1/2}
\\
&\le \left (\sum_ {a_ k\in \Gamma (\zeta)}  |\beta_ k|^2 \right )^{1/2}\left (\sum_ {a_ k\in \Gamma (\zeta)}  |\lambda_ k|  \widehat{\mu}_{2t}(a_ k)^{-\sigma/2}\,\widehat{\mu}_ t(a_ k) \right )^{1/2}.
\end{split}
\]
Finally, an application of H\"{o}lder's inequality with exponent $4mp/(2mp+\sigma)$ that has conjugate exponent $4mp/(2mp-\sigma)$, together with our condition \eqref{EqT-02}  shows that $\widetilde{\mu}$ is in $L^{r/s}(\Sn)$. An examination of the estimates as done in the proof of Proposition \ref{pro-1} gives \eqref{iEqT-10}.
\end{proof}
\mbox{}
\\
Now we are ready for the proof of the necessity for the case $r<s$ in Theorem \ref{MT2}.

\begin{theorem}\label{MT2-b}
Let $\mu$ be a positive Borel measure on $\Bn$. Let $0<q<p<\infty$,  $s>0$  with $r/s<1$ where $r=pq/(p-q)$. If $A_{\mu, s}:H^p\rightarrow L^q(\Sn)$ is bounded, then $\widetilde{\mu} \in L^{r/s}(\Sn)$. Moreover, we have $\big \|\widetilde{\mu} \big \|_{L^{r/s}(\Sn)} \lesssim \big \| A_{\mu,s}  \big \|^s_{H^p\rightarrow L^q(\Sn)}$.
\end{theorem}

\begin{proof}
 Take a positive integer $m$ big enough so that \eqref{EqT-10} holds. From Lemma \ref{L2}, we know that  $A_{\mu,2ms}:H^{2mp}\rightarrow L^{2qm}(\Sn)$ is bounded with $\|A_{\mu,2ms}\|_{H^{2mp}\rightarrow L^{2mq}(\Sn)}\le \|A_{\mu,s}\|_{H^p \rightarrow L^q(\Sn)}^{1/2m}$.
Set $\sigma$ as in \eqref{EqT-04}.  We have
\[
\frac{mp(2-\sigma)}{2(mp-\sigma)}=r/s=\frac{pq}{s(p-q)}.
\]
Observe that, as $r/s<1$, $mp>2$ and  \eqref{EqT-10}, it follows that $0<\sigma<2$.
We also notice that
\begin{equation}\label{Eq1-q1}
q=\frac{(2-\sigma)mps}{2mp+ms(2-\sigma)-2\sigma}.
\end{equation}
Now, we proceed to establish the inequality
\begin{equation}\label{MIn}
\int_{\Bn \cap spt (\mu)} \!\!\!\!\! |g(z)|^{\sigma} \,|Rf(z)|^{\sigma} \, \widehat{\mu}_{t}(z)^{-\sigma/2}(1-|z|^2)^{\sigma}\,d\mu(z) \lesssim \|g\|_{H^{2mp}}^{\sigma} \cdot \|f\|^{\sigma}_{H^{2mp}}.
\end{equation}
Using \eqref{EqG}, H\"{o}lder's inequality with exponent $2/\sigma>1$, and the estimate \eqref{Eq-RS11}, we get
\[
\begin{split}
\int_{\Bn \cap spt(\mu)}\!\!\!\!\!& |g(z)|^{\sigma} \,|Rf(z)|^{\sigma} \,  \widehat{\mu}_{t}(z)^{-\sigma/2}\,(1-|z|^2)^{\sigma}\,d\mu(z)
\\
& \lesssim \int_{\Sn} \! \left (\int_{\widetilde{\Gamma}(\zeta)} \!\!\! |Rf(z)|^2\,dv_{1-n}(z) \right )^{\sigma/2} \!\!\left (\int_{\Gamma(\zeta)} \!\!\!|g(z)|^{\frac{2\sigma}{2-\sigma}} \,\frac{d\mu(z)}{(1-|z|^2)^n} \right )^{1-\sigma/2}\!\!\!\!\!\!\!\!\!d\sigma (\zeta).
\end{split}
\]
Now, by H\"{o}lder's inequality with exponent $2mp/\sigma$ and the area function description of Hardy spaces, we get
\begin{equation}\label{Eqrs25}
\begin{split}
\int_{\Bn \cap spt (\mu)} |g(z)|^{\sigma} &\,|Rf(z)|^{\sigma} \,\widehat{\mu}_ t(z)^{-\sigma/2} (1-|z|^2)^{\sigma}\,d\mu(z)
\\
& \lesssim \|f\|_{H^{2mp}}^{\sigma} \left(\int_{\Sn} A_{\mu,\frac{2\sigma}{2-\sigma}} g(\zeta) ^{2mp\sigma/(2mp-\sigma)}\,d\sigma(\zeta)\right )^{\frac{(2mp-\sigma)}{2mp}}.
\end{split}
\end{equation}
If $\sigma=(2-\sigma)ms$, by \eqref{Eq1-q1} we have $2mp\sigma/(2mp-\sigma)=2mq$, and we get
\[
\begin{split}
\int_{\Bn \cap spt (\mu)} |g(z)|^{\sigma}\,& |Rf(z)|^{\sigma} \,\widehat{\mu}_ t(z)^{-\sigma/2} (1-|z|^2)^{\sigma}\,d\mu(z)
\\
& \lesssim \|f\|_{H^{2mp}}^{\sigma} \cdot \big \|A_{\mu,2ms} g \big \|^{\sigma}_{L^{2mq}(\Sn)}.
\end{split}
\]
Thus, as $A_{\mu,2ms}:H^{2mp}\rightarrow L^{2mq}(\Sn)$ is bounded with $\|A_{\mu,2ms}\|_{H^{2mp}\rightarrow L^{2mq}(\Sn)}\le \|A_{\mu,s}\|_{H^p \rightarrow L^q(\Sn)}^{1/2m}$, by Proposition \ref{L1-A} we get $\widetilde{\mu}\in L^{r/s}(\Sn)$ with $\|\widetilde{\mu}\|_{L^{r/s}(\Sn)}\lesssim \|A_{\mu,s}\|_{H^p \rightarrow L^q(\Sn)}^s$, and we are done. \\
\mbox{}
\\
 If $\sigma>(2-\sigma)ms$, then
\[
A_{\mu,\frac{2\sigma}{2-\sigma}} g(\zeta) \le g^{*}(\zeta)^{\frac{\sigma-(2-\sigma)ms}{\sigma}} \cdot A_{\mu,2m} g(\zeta)^{\frac{m(2-\sigma)}{\sigma}}.
\]
Since $mp>2$ and $0<\sigma<2$, it is easy to check that
\[
\frac{2mp-\sigma}{\sigma-(2-\sigma)ms}>1,
\]
that has conjugate exponent
\[
\left (\frac{2mp-\sigma}{\sigma-(2-\sigma)ms}\right )'=\frac{2mp-\sigma}{2mp+(2-\sigma)ms-2\sigma}=\frac{(2mp-\sigma)q}{(2-\sigma)mps}.
\]
Thus, we can use H\"{o}lder's inequality with that exponents in order to obtain
\[
\begin{split}
\int_{\Sn} A_{\mu,\frac{2\sigma}{2-\sigma}}& g(\zeta) ^{2mp\sigma/(2mp-\sigma)}\,d\sigma(\zeta) 
\\
&\le \int_{\Sn} g^*(\zeta)^{2mp \cdot \frac{[\sigma-ms(2-\sigma)]}{2mp-\sigma}}A_{\mu,2ms} g(\zeta) ^{\frac{2m^2ps(2-\sigma)}{2mp-\sigma}}\,d\sigma(\zeta)
\\
& \le \left ( \int_{\Sn} g^*(\zeta)^{2mp}\,d\sigma(\zeta)\right )^{\frac{\sigma-ms(2-\sigma)}{2mp-\sigma}} \left ( \int_{\Sn} A_{\mu,2ms} g(\zeta) ^{2mq}\,d\sigma(\zeta) \right )^{\frac{(2-\sigma)mps}{(2mp-\sigma)q}}.
\end{split}
\]
Putting this in the estimate \eqref{Eqrs25}, by the $L^p$-boundedness of the admissible maximal function, we have
\[
\begin{split}
\int_{\Bn \cap spt (\mu)} |g(z)|^{\sigma}\,|Rf(z)|^{\sigma} &\,\widehat{\mu}_ t(z)^{-\sigma/2} (1-|z|^2)^{\sigma}\,d\mu(z)
\\
& \lesssim \|f\|_{H^{2mp}}^{\sigma}\cdot \|g \|_{H^{2mp}}^{\sigma-ms(2-\sigma)}\cdot \big \| A_{\mu,2ms} g\big \|^{(2-\sigma)ms} _{L^{2mq}(\Sn)},
\end{split}
\]
and again the result follows by the boundedness of $A_{\mu,2ms}:H^{2mp}\rightarrow L^{2mq}(\Sn)$ and Proposition \ref{L1-A}.\\
\mbox{}
\\
If $\sigma<(2-\sigma)ms$, then
by H\"{o}lder's inequality with exponent $ms(2-\sigma)/\sigma$, we have
\[
 A_{\mu,\frac{2\sigma}{2-\sigma}} g (\zeta) \le A_{\mu,2ms} g(\zeta) \,\,\widetilde{\mu}(\zeta)^{\frac{ms(2-\sigma)-\sigma}{2\sigma ms }}.
\]
Also, as $\sigma<2$ and $mq>2$, we can see that
\[
\frac{q(2mp-\sigma)}{p\sigma}>1,
\]
and its conjugate exponent is
\[
\left (\frac{q(2mp-\sigma)}{p\sigma}\right )'=\frac{q(2mp-\sigma)}{q(2mp-\sigma)-p\sigma}.
\]
It is possible to check that
\[
\frac{pq[ms(2-\sigma)-\sigma]}{q(2mp-\sigma)-p\sigma}=r.
\]
Hence, applying H\"{o}lder's inequality with the previous exponents we have
\[
\begin{split}
\int_{\Sn} A_{\mu,\frac{2\sigma}{2-\sigma}} g(\zeta) ^{\frac{2mp\sigma}{2mp-\sigma}}&\,d\sigma(\zeta) \le \int_{\Sn}A_{\mu,2ms} g(\zeta) ^{\frac{2mp\sigma}{2mp-\sigma}} \,\,\widetilde{\mu}(\zeta)^{\frac{p[ms(2-\sigma)-\sigma]}{s(2mp- \sigma) }}\,d\sigma(\zeta)
\\
&
\le \Big \| A_{\mu,2ms} g \Big \|_{L^{2mq}(\Sn)}^{\frac{2mp\sigma}{2mp-\sigma}} \cdot \big \|\widetilde{\mu} \big \|_{L^{r/s}(\Sn)}^{\frac{p[ms(2-\sigma)-\sigma]}{s(2mp- \sigma) }}.
\end{split}
\]
Using Proposition \ref{L1-A} we get
\[
\begin{split}
\|\widetilde{\mu}\|^{1-\sigma/2}_{L^{r/s}(\Sn)}&\asymp \sup \int_{\Bn} |g(z)|^{\sigma} \,|Rf(z)|^{\sigma} \, \widehat{\mu}_{t}(z)^{-\sigma/2}\,(1-|z|^2)^{\sigma}\,d\mu(z)
\\
& \lesssim  \Big \| A_{\mu,2ms}  \Big \|_{H^{2mp} \rightarrow L^{2mq}(\Sn)}^{\sigma} \cdot \big \|\widetilde{\mu} \big \|_{L^{r/s}(\Sn)}^{\frac{ms(2-\sigma)-\sigma}{2ms }}
\\
& \lesssim \Big \| A_{\mu,s}  \Big \|_{H^p \rightarrow L^q(\Sn)}^{\sigma/2m} \cdot \big \|\widetilde{\mu} \big \|_{L^{r/s}(\Sn)}^{\frac{ms(2-\sigma)-\sigma}{2ms }},
\end{split}
\]
and, assuming that $\widetilde{\mu}$ is already in $L^{r/s}(\Sn)$, this shows that
\[
\big \|\widetilde{\mu} \big \|_{L^{r/s}(\Sn)} \lesssim \big \| A_{\mu,s}  \big \|_{H^p \rightarrow L^q(\Sn)}^s.
\]
Finally, the case for general $\mu$ follows from this estimate and an standard approximation argument, finishing the proof of the theorem.
\end{proof}

{\bf Acknowledgments:} \  Most of this work was done while the first author  visited Universitat de Barcelona in 2019.

% ------------------------------------------------------------------------


\begin{thebibliography}{99}

\bibitem{AB} P. Ahern \and J. Bruna, \emph{Maximal and area integral characterizations of Hardy-Sobolev spaces in the unit ball of $\Cn$}, Rev. Mat. Iberoamericana 4 (1988), 123--153.

\bibitem{Alek} A. Aleksandrov, `Function Theory in the Ball', in \emph{Several Complex Variables II} (G.M. Khenkin \and A.G. Vitushkin, editors), 115--190, Springer-Verlag, Berlin, 1994.

\bibitem{Am} A. Amenta, \emph{Interpolation and embeddings of weighted tent spaces}, J. Fourier Anal. Appl. 24 (2018), 108--140.

\bibitem{Ars} M. Arsenovic, \emph{Embedding derivatives of $\mathcal{M}$-harmonic functions into $L^p$ spaces}, Rocky Mountain J. Math. 29 (1999), 61--76.

\bibitem{Cal} A. Calder\'{o}n, \emph{Commutators of singular integral operators}, Proc. Nat. Acad. Sci. USA 53 (1965), 1092--1099.

\bibitem{Carleson-0} L. Carleson, \emph{An interpolation problem for bounded analytic functions}, Amer. J. Math. 80 (1958), 921--930.

\bibitem{carleson} L. Carleson,
\emph{Interpolation by bounded analytic functions and the corona problem},
 Ann. of Math. 76 (1962), 547--559.

\bibitem{C-O} C. Cascante \and J.M. Ortega, \emph{Imbedding potentials in tent spaces} J. Funct. Anal. 198 (2003), 106–-141.

\bibitem{Chang} S-Y. A. Chang, \emph{Carleson measure on the bi-disc}, Ann. of Math. 109 (1979), 613--620.

\bibitem{Cohn} W. Cohn, \emph{Generalized area operators on Hardy spaces}, J. Math. Anal. Appl. 216 (1997), 112--121.


\bibitem{CMS} {R.} Coifman, Y. Meyer \and {E.} Stein, \emph{ Some new {function} spaces and their applications to Harmonic Analysis}, J. Funct. Anal. 62 (1985), 304--335.

\bibitem{Du} P. Duren, {\em{Extension of a theorem of Carleson}}, Bull. Amer. Math. Soc. 75 (1969), 143–-146.


\bibitem{RFef} R. Fefferman, \emph{A note on Carleson measures in product spaces}, Proc. Amer. Math. Soc. 93 (1985), 509--511.

\bibitem{FS} C. Fefferman \and E. Stein, \emph{$H^p$ spaces of several variables}, Acta Math. 129 (1972), 137--193.

\bibitem{GLW} M. Gong, Z. Lou \and Z. Wu, \emph{Area operators from $H^p$ spaces to $L^q$ spaces}, Sci. China Math. 53 (2010), 357--366.

\bibitem{H} L. H\"{o}rmander, \emph{$L^p$ estimates for (pluri-)subharmonic functions}, Math. Scand. 20 (1967), 65--78.

\bibitem{Jev} M. Jevtic, \emph{Embedding derivatives of $\mathcal{M}$-harmonic Hardy spaces $\mathcal{H}^p$ into Lebesgue spaces, $0<p<2$}, Rocky Mountain J. Math. 26 (1996), 175--187.
\bibitem{L-RP} P. Lef\`{e}vre \and L. Rodr\'{i}guez-Piazza, \emph{Absolutely summing Carleson embeddings on Hardy spaces}, Adv. Math. 340 (2018), 528--587.

\bibitem{Lue1} D. Luecking, \emph{Embedding derivatives of Hardy spaces into Lebesgue spaces}, Proc. London Math. Soc. 63 (1991), 595--619.

\bibitem{MPPW19} S. Miihkinen, J. Pau, A. Per\"{a}l\"{a} \and  M. F. Wang,  \emph{Volterra type integration operators from Bergman spaces to Hardy spaces}, J. Funct. Anal. 279 (2020), 108564, 32pp.

\bibitem{P1} J. Pau, \emph{Integration operators between Hardy spaces on the unit ball of $\mathbb{C}^n$}, J. Funct. Anal. 270 (2016), 134--176.

\bibitem{PP1} J. Pau \and A. Per\"{a}l\"{a}, \emph{A Toeplitz type operator on Hardy spaces in the unit ball}, Trans. Amer. Math. Soc. 373 (2020), 3031--3062.

\bibitem{PR2} J.A. Pel\'{a}ez \and J. R\"{a}tty\"{a}, \emph{Embedding theorems for Bergman spaces via harmonic analysis}, Math. Ann. 362 (2015), 205--239.

\bibitem{PRS} J.A. Pel\'{a}ez, J. R\"{a}tty\"{a} \and K. Sierra, \emph{Embedding Bergman spaces into tent spaces}, Math. Z. 281 (2015), 1215--1237.


\bibitem{Rud} W. Rudin, `Function Theory in the Unit Ball of $\mathbb{C}^n$', Springer, New York, 1980.

\bibitem{SaSa} I. Sabadini, \and A. Saracco, \emph{Carleson measures for Hardy and Bergman spaces in the quaternionic unit ball}, J. Lond. Math. Soc. (2) 95 (2017), 853–-874.

\bibitem{WuIEOT} Z. Wu, \emph{A new characterization for Carleson measures and some applications}, Integr. Equ. Oper. Theory 71 (2011), 161--180.


\bibitem{Wu06} Z. Wu, \emph{Area operator on Bergman spaces}, Sci. China Ser. A 49 (2006), no. 7, 987--1008.

\bibitem{ZZ} R. Zhao \and K. Zhu, \emph{Theory of Bergman spaces in the
unit ball of $\mathbb{C}^n$},  Mem. Soc. Math. Fr. 115,
2008.
\bibitem{ZhuBn} K. Zhu, `Spaces of Holomorphic Functions in the Unit Ball', Springer-Verlag, New York, 2005.




\end{thebibliography}
\end{document}